\documentclass[10pt, reqno]{amsart}
 \usepackage{tikz}
 \usepackage{hyperref, mathtools}
 \usepackage[skip=4pt]{caption}
 \usepackage{amssymb,amsmath,graphicx,amsthm}
  \usetikzlibrary{positioning}
  \usepackage{tkz-euclide}
\usepackage{rotating}
\usepackage[pagewise]{lineno} %\linenumbers
\usepackage{color,xcolor}
\definecolor{darkred}{rgb}{1,0,0} %can change the intensity in [0,1]
\definecolor{darkgreen}{rgb}{0,0.8,0}
\definecolor{darkblue}{rgb}{0,0,1}
\newcommand{\norm}[1]{\left\lVert#1\right\rVert}
\hypersetup{colorlinks,
linkcolor=darkblue,
filecolor=darkgreen,
urlcolor=darkred,
citecolor=darkgreen}

 \def\bt{\begin{theorem}}
 	\def\el{\end{lemma}}
 \def\bl{\begin{lemma}}
 	\def\et{\end{theorem}}
 \def\bp{\begin{proposition}}
 	\def\ep{\end{proposition}}
 \def\bd{\begin{definition}}
 	\def\ed{\end{definition}}
 \def\br{\begin{remark}}
 	\def\er{\end{remark}}

 \def\R{{\mathbb R}}

 \def\C{\mathbb C}
 
 \def\Z{{\mathbb Z}}
 \def\N{{\mathbb N}}
 \def\P{{\mathbb P}}
\def\S{{\mathbb S}}
\def\S1{{\mathbb S^1}}

 \def\label#1{\label{#1}{\bf(#1)}~}

 \numberwithin{equation}{section}

 \theoremstyle{plain}
 \newtheorem*{theorem*}{Theorem}

 \newtheorem{theorem}{Theorem}[section]
 
 \newtheorem{corollary}[theorem]{Corollary}
 
 \newtheorem{lemma}[theorem]{Lemma}
 
 \newtheorem{proposition}[theorem]{Proposition}

 \theoremstyle{definition}
 
 \newtheorem{definition}[theorem]{Definition}
 
 \theoremstyle{remark}

 \newtheorem{remark}[theorem]{Remark}

 \DeclareMathOperator{\const}{const}

 \DeclareMathOperator{\ind}{ind}

  \DeclareMathOperator{\Aut}{Aut}

 \newtheoremstyle{named}{}{}{\itshape}{}{\bfseries}{.}{.5em}{\thmnote{#3 }#1}
\theoremstyle{named}
\newtheorem*{namedtheorem}{Theorem}

 \theoremstyle{plain} % just in case the style had changed

 \def\Re{\text{Re}}

\begin{document}
 	
 	\title{On the Ekeland-Hofer symplectic capacities of the real bidisc }  
\author{Luca Baracco}
\address{Dipartimento di Matematica Tullio Levi-Civita, Universit\`a di Padova, via Trieste 63, 35121 Padova, Italy}
\email{baracco@math.unipd.it}

\author{Martino Fassina}
\address{Department of Mathematics, University of Illinois, 1409 West Green
Street, Urbana, IL 61801, USA}
\email{fassina2@illinois.edu}

\author{Stefano Pinton}
\address{Dipartimento di Matematica Tullio Levi-Civita, Universit\`a di Padova, via Trieste 63, 35121 Padova, Italy}
\email{pinton@math.unipd.it}
\thanks{M.F. acknowledges support of NSF grant 13-61001.}
       % Enter your title between curly braces
 %\author[Khanh]{Tran Vu Khanh}
 	%	\address{Tran Vu Khanh}
 		%\address{School of Mathematics and Applied Statistics, University of Wollongong, NSW, Australia,  2522}
 %		\email{tkhanh@uow.edu.au}
 	
 	%	\thanks{Research was partially supported by the Australian Research Council
 		%	DE160100173. This work was done in part while the author was  a visiting member at the Vietnam Institute for Advanced Study in Mathematics %(VIASM). He would like to thank the institution for its hospitality and support.}

 \begin{abstract} 
In $\C^2$ with the standard symplectic structure we consider the bidisc $D^2\times D^2$ constructed as the product 
of two open real discs of radius $1$. We compute explicit values for the first, second and third Ekeland-Hofer symplectic capacity of $D^2\times D^2$. We discuss some applications to questions of symplectic rigidity.
\end{abstract}
\subjclass[2010]{Primary 	53D05.  
Secondary 53D35}
\keywords{Ekeland-Hofer symplectic capacities, Lagrangian bidisc.}
  	\maketitle
  	\maketitle

 %***********************************
 % Introduction 
 %***********************************
% \tableofcontents 	
 	\section{Introduction and Main Result}

The first striking result about a non-trivial obstruction to the existence of a symplectic embedding was obtained by Gromov \cite{G85}. 
He proved that one can symplectically embed a sphere into a cylinder 
only if the radius of the sphere is less than or equal to the radius of the cylinder. Since this celebrated non-squeezing theorem appeared, many similar results of symplectic rigidity have been obtained for a variety 
of domains. For instance, McDuff studied symplectic embeddings of even-dimensional open ellipsoids into one another \cite{McD09,McD11} 
(see also Hutchings \cite{H11,H111} and McDuff-Schlenk \cite{McDS12}), and Guth \cite{G08} gave asymptotic results on when a complex polydisc can be symplectically embedded into another one. 
A useful tool to tackle these questions is given by global symplectic invariants for symplectic manifolds called capacities. 

A symplectic capacity is a functor $c$ that assigns to every symplectic manifold $(M,\omega)$ of dimension $2n$ a non-negative (possibly infinite) number $c(M,\omega)$ that satisfies the following conditions: 
\begin{enumerate}
\item (Monotonicity) If there exists a symplectic embedding of $(M_1,\omega_1)$ into $(M_2,\omega_2)$, then 
$c(M_1,\omega_1)\leq c(M_2,\omega_2).$
\item (Conformality) If $\lambda>0$, then $c(M,\lambda\omega)=\lambda c(M,\omega)$.
\item (Local nontriviality) For the open unit ball $B\subset\R^{2n}$ we have $c(B,\omega_0)>0$.
\item (Nontriviality) For the open cylinder $Z=\{z\in\C^n\colon |z_1|<1\}$ we have $c(Z,\omega_0)<\infty$.
\end{enumerate} 
Here $\omega_0$ denotes the standard symplectic structure on $\R^{2n}$. The reader can consult any standard textbook in the subject such as \cite[Chapter 12]{McDS17} for an extensive treatment of this topic. 

The first non-trivial capacity arose in Gromov's proof of the non-squeezing theorem, and over the years many more such symplectic invariants have been constructed by Ekeland and Hofer \cite{EH89,EH90}, Hofer and Zender \cite{HZ90}, Viterbo and others \cite{FHV90,V92}, Gutt and Hutchings \cite{GH18}.
 
In this paper we consider the construction of Ekeland and Hofer (which is recalled in Section \ref{SecHofer}). For any subset of a symplectic vector space, they define an infinite sequence $c_n$ of symplectic capacities. These quantities are notoriously difficult to compute explicitely, and precise values appear in the literature only for very special classes of domains, such as ellipsoids and polydiscs \cite{EH90}. The main purpose of this work is to compute some of these capacities for the real bidisc $D^2\times D^2$, which we now introduce.

In the complex space $\C^2$ with coordinates $z_j=x_j+iy_j, j=1,2$ endowed with the standard symplectic structure $dx_1\wedge dy_1+ dx_2\wedge dy_2$, consider the real bidisc
 \begin{equation}\label{realbidisc}
 D^2\times D^2:=\big{\{}(z_1,z_2)\in \C^2 | \  x_1^2+x^2_2 <1, y^2_1+y^2_2<1 \big{\}}.
 \end{equation}
Our main result is the following (see Theorems \ref{c1}, \ref{c2bis} and \ref{cap3}).
 
\begin{namedtheorem}[Main]\label{nine} 
For the unit real bidisc $D^2\times D^2$ we have $$c_1(D^2\times D^2) =4, \,\,c_2(D^2\times D^2)=3\sqrt{3},\,\, c_3(D^2\times D^2)=8.$$
\end{namedtheorem}

 The fact that $c_1(D^2\times D^2) =4$ is known \cite{AO14}. The referee pointed out to us that the values of $c_2(D^2\times D^2)$ and $c_3(D^2\times D^2)$ can also be obtained from the work of Ramos \cite{R17}, which uses very different techniques than the ones developed in this paper (see Remark \ref{rmkRamos} and Remark \ref{R2}).

The Ekeland-Hofer capacities are closely related to the existence of closed Hamiltonian orbits. The key feature of our computation is the use of Hamiltonians modeled on the gauge function of the domain, rather than Hamiltonians that are quadratic at infinity, as in the original definition of Ekeland and Hofer. This same idea already appears in \cite{BLMR85} in the context of smooth domains, and in \cite{F92} for the case of Lipschitz domains. The different choice of Hamiltonians does not affect the computation of the capacities, provided that no periodic orbits of period 1 are introduced at infinity. This feature is easily achieved by rescaling, since the closed periodic orbits of the Hamiltonians in question can be computed explicitely. The advantage in modeling the Hamiltonians on the gauge function is that we can then exploit the symmetries of the domain to simplify the estimates involved in the computations of the capacities.

We remark that our strategy can be easily adapted to compute the Ekeland-Hofer capacities of other domains. For instance, for the product of two real spheres in $\C^3$ one just needs to repeat the same computations while taking an additional coordinate into account.
 
We now point out how the results of this paper are related to recent work of Gutt and Hutchings. In \cite{GH18}, they use positive $S^1$-equivariant symplectic homology to introduce a sequence $c_k$ of symplectic capacities for star-shaped domains in $\R^{2n}$. Their capacities $c_k$ share many properties with the Ekeland-Hofer capacities, and in particular the remarkable ``product property" (see Section \ref{SectionMain}). Combining the results of \cite{GH18} with Ramos' insight that the real bidisc is a toric domain \cite{R17}, one can see that the Gutt-Hutchings capacities of $D^2\times D^2$ are the same as the Ekeland-Hofer capacities that we obtain in our Main Theorem. Our computations therefore support the conjecture made by Gutt and Hutchings \cite{GH18} that their capacities $c_k$ are always equal to the Ekeland-Hofer capacities. 
 
The organization of this paper is as follows. In Section \ref{1} we recall some background material and set the notation. In Section \ref{last} we show how one can approximate the bidisc $D^2\times D^2$ with a sequence of smooth convex domains. The main result is then proved in Section \ref{SectionCap}, while in Section \ref{SectionMain} we present some applications to questions of symplectic embedding.
It is the authors' hope that further applications to symplectic rigidity of these explicit values of the capacities will be found in the future.

 	\section{Background}\label{1}
 	\subsection{Basic definitions and notations}\label{sub1}
 	
Let $\C^n$ denote the standard complex vector space of dimension $n$ with variables $z_j=x_j+iy_j$. We endow $\C^n$ with the Euclidean scalar product $$\langle z,w\rangle := \Re \left(\sum^n_{j=1}z_j\bar{w_j}\right)$$ 
and the standard symplectic form $$\omega:=\frac 1{2i}\sum^n_{j=1}dz_j\wedge d\bar{z_j}=\sum^n_{j=1}dx_j\wedge dy_j.$$ All symplectic embeddings considered in this paper will be with respect to the standard symplectic form.

Let $\Omega$ be a bounded subset of $\C^n$ with smooth boundary $\partial\Omega$. We denote by $T\C^n$ the (real) tangent bundle of $\C^n$ and by $T\partial\Omega$ the tangent bundle of $\partial\Omega$. We write $Y\in T\C^n$ to mean that $Y$ is a local section of $T\C^n$. Consider a smooth defining equation $\rho$ of $\partial\Omega$ such that $|d\rho (z)|\neq 0$ for $z\in \partial\Omega$.
The {\em characteristic vector field} of $\partial\Omega$ is the unique vector field $X$ such that
\begin{equation*}
\omega(X,Y)=Y(\rho) \text{ for all } Y\in T\C^n.
\end{equation*}  
The characteristic vector field $X$ is tangent to $\partial\Omega$ and its restriction to $\partial\Omega$ does not depend on the choice of $\rho$. Moreover, $X$ generates the kernel of the restricted form $\omega_{|T\partial\Omega}$.

Let $J$ denote the standard complex structure on $\C^n$. For all $z\in\partial\Omega$ we have $X(z)=J\nabla\rho(z)$, where $\nabla \rho$ is the Euclidean gradient of $\rho$. Hence an integral curve of $X$ on $\partial \Omega$ is the solution of a system
\begin{equation}
\label{sistema}
\begin{cases} 
\dot{z} =J\nabla\rho(z)\\
z(0)=z_0
\end{cases}
\end{equation} 
where $z_0\in\partial\Omega$. These integral curves are called the {\em characteristics} of $\partial\Omega$. 

Of particular interest in symplectic geometry is the study of closed characteristics, that is, the solutions to \eqref{sistema} for which there exists a time $t_0>0$ such that $z(t_0)=z_0$. Let $T$ be the smallest such $t_0$. The image $\{z(t), \, 0\leq t\leq T\}$ is called an {\em orbit} and $T$ the {\em period} of the orbit.

If $\gamma:[0,T]\rightarrow \partial\Omega$ is a closed characteristic, the {\em action} of $\gamma$ is defined to be
$$\mathcal{A}(\gamma):=-\frac 12 \int^T_0 \langle J\dot{\gamma},\gamma\rangle\, dt.$$
If $\Omega$ is a bounded convex subset of $\C^n$ and $\{\gamma_i\}_{i\in I}$ is the set of closed characteristics of $\partial\Omega$, we can define the {\em action spectrum} of $\Omega$ as the set
$$\Sigma(\Omega):=\big{\{} |k\mathcal{A}(\gamma_i)|, k\in\mathbb{N}, i\in I\big{\}}.$$
Following Ekeland and Hofer \cite{EH90}, we will see how it is possible to choose some elements of $\Sigma(\Omega)$ called capacities that are symplectic invariants (see Section \ref{SecHofer}). We now describe how to adapt the concepts introduced above to non-smooth domains. 

Let $\Omega\subset \C^n$ be a convex bounded domain and $p\in \partial \Omega$. We say that a unit vector $n(p)$ is {\em normal} at $p$ for $\partial \Omega$ if 
\begin{equation} \label{normal} \langle n(p),x-p\rangle \le 0 \quad \forall x\in \Omega .\end{equation}
If $p$ is a smooth boundary point, then $n(p)$ is the usual exterior normal vector. If $\partial \Omega$ is not smooth at $p$, then there could be more than one choice for a normal vector. In this case, we let $n(p)$ denote the set of all vectors satisfying \eqref{normal}.

\bd Let $\Omega$ be a convex domain in $\C^n$. We say that $z\colon\R\rightarrow \Omega$ is a {\em characteristic} if $z(t)$ has right and left derivative $\dot{z}^{\pm}(t)$ for all $t$, and
\begin{equation*}
\dot{z}^{\pm}(t) \in Jn(z(t)). 
\end{equation*}
\ed Note that, for $\Omega$ smooth, this definition coincides with the definition of a characteristic given before.

\subsection{The action spectrum of the real bidisc}

We now turn our attention to the domain that is the main object of interest of this paper: the real bidisc $D^2\times D^2$. Recall that $D^2\times D^2$ was defined in \eqref{realbidisc} as the product of two open real discs of radius $1$. The next proposition, which follows from the work of Artstein-Avidan and Ostrover \cite{AO14}, describes the closed characteristics of $D^2\times D^2$ and its action spectrum $\Sigma(D^2\times D^2)$. 
\bp\label{s1} The unitary real bidics $D^2\times D^2$ has infinitely many closed characteristics. The action spectrum is given by the following set
\begin{equation} \label{s3} \Sigma(D^2\times D^2)=\big{\{} 2n\cos(\theta_{k,n})\,\vert\ k,n\in \N,\theta_{k,n}\in J_n\big{\}}\cup\big{\{}2n\pi\,\vert \ n\in\N \big{\}}, \end{equation}
where $J_n=\big{\{}\frac{k\pi}{n}, 0\le k\le \frac n2 -1\big{\}}$ if $n$ is even and $J_n=\big{\{}\frac {(2k-1)\pi}{2n}, 1\le k\le \frac{n-1}2 \big{\}}$ if n is odd.
\ep

\br The elements of $\Sigma(D^2\times D^2)$ are precisely the lengths of all closed billiard orbits in a circle of radius 1. 
\er
\br\label{rmk}Note that $\min\Sigma(D^2\times D^2)=4$. Moreover, the second smallest element in $\Sigma(D^2\times D^2)$ is $3\sqrt{3}$.
\er
\subsection{Ekeland-Hofer symplectic capacities}\label{SecHofer}
We recall, following \cite{EH90}, the definition of these symplectic capacities. We first set up the functional analytical framework. For more details see also \cite[Section II]{EH89}. 

Let $E$ be the Hilbert space of all functions $f\in L^2(\mathbb{R}/\mathbb{Z},\mathbb{C}^n)$ such that the Fourier series
$$f(t)=\sum_{k\in \Z} f_ke^{2k\pi i t},\quad f_k\in\mathbb{C}^n$$
satisfies
$$\sum_{k\in\Z}|k||f_k|^2 <\infty.$$
The inner product in $E$ is defined by
$$(f,g):=\langle f_0,g_0\rangle + 2\pi\sum_{k\in\mathbb{Z}}|k|\langle f_k,g_k\rangle.$$
%nuovo
$E$ is the most natural space on which the action functional $\mathcal{A}$ can be defined. It is easy to see that 
$$\mathcal{A}(f)=\pi \sum^{+\infty}_{k=0} k(|f_k|^2-|f_{-k}|^2).$$ 
Note that there is a natural action $T\colon \mathbb{S}^1\rightarrow \Aut(E)$ of $\mathbb{S}^1\simeq \R/\Z$ on $E$ given by the phase shift
$$ T_{e^{2\pi i\theta}} f(t):=f(t+\theta).$$
%We denote by $\norm{\cdot}$ the corresponding norm.
The space $E$ has a natural orthogonal splitting, compatible with the phase shift action, given by $$E=E^{-}\oplus E^0 \oplus E^{+}.$$ Here we have defined
\begin{equation*}
\begin{split}
&E^{-}=\{f\in E\,|\,f_k=0\text{ for } k\geq 0 \}\\
&\,E^{0}=\{f\in E\,|\,f_k=0\text{ for } k\neq 0 \}=\C^n \\
&E^{+}=\{f\in E\,|\,f_k=0\text{ for } k\leq 0 \}.\\
\end{split}
\end{equation*}
We denote by $P^+, P^0$ and $P^-$ the corresponding orthogonal projections. 

We now need to introduce the notion of the index of a subspace. Let $X$ be a Hilbert space over $\C$, and let $T\colon S^1\rightarrow \Aut(X)$ be a representation of $S^1$ on the vector space $$\Aut(X):=\{f\colon X\rightarrow X \,|\text{ $f$ is a linear isometry}\}.$$ A subset $A\subseteq X$ is called {\em invariant} if $T(\theta)(A)=A$ for all $\theta\in S^1$. 
Let $Y$ be another Hilbert space, and $R\colon S^1\rightarrow \Aut(Y)$ a representation of $S^1$ on $\Aut(Y)$. A linear map $f\colon X\rightarrow Y$ is called {\em invariant} if $f\circ T(\theta)=R(\theta)\circ f$ for all $\theta\in S^1$. Let $A\subseteq X$ be an invariant subset. For every $k\in\mathbb{N}$, we let $\mathcal{F}(A,k)$ denote the collection of functions $f\colon A\rightarrow \C^k\setminus\{0\}$ such that 
\begin{itemize}
\item $f$ is continuous.
\item There exists a positive integer $n$ such that $f(T(\theta)(x))=e^{2\pi n i\theta}f(x)$ for all $\theta\in S^1$ and for all $x\in A$.
\end{itemize}
We define the {\em index} of $A$ as the quantity 
\begin{equation}\label{index}
\alpha(A):=\min\{k\in\mathbb{N}\,|\, \mathcal{F}(A,k)\neq\varnothing\}.
\end{equation}
If $\mathcal{F}(A,k)=\varnothing$ for every $k\in\mathbb{N}$, we set $\alpha(A)=+\infty$. Moreover, we set $\alpha(\varnothing)=0$. 
Observe that if $F$ is the set of fixed points of $X$ for $T$, that is, $$F:=\{x\in X\,|\, T(\theta)(x)=x\,\text{ for all }\theta\in S^1\},$$ then $A\cap F\neq\varnothing$ implies $\alpha(A)=\infty$. %This observation tells us that the index is not 
%enough for our purposes \textcolor{red}{Magari dire qualcosa di piu qui} and motivates 

In the following paragraph we describe a pseudoindex theory in the sense of Benci relative to $E^+$. In our exposition we mainly follow \cite{B82} (see also \cite{EH90}). 

%First note that there is a natural representation $T$ of $S^1$ on $\Aut(E)$, defined by 
%$$T(\theta)(x)(t)=x(t+\theta)\quad \text{for $x\in E$, $\theta\in S^1$}.$$
Consider the group of homeomorphisms
\begin{equation}\label{groupGamma}
\Gamma:=\big{\{}h\colon E\rightarrow E\,|\, h=e^{\gamma^+}P^{+}+e^{\gamma^-}P^{-}+P^0+K\big{\}}
\end{equation}
such that the following conditions are satisfied:
\begin{itemize}
\item $K$ is a compact equivariant operator. 
\item $\gamma^{+}, \gamma^{-}\colon E\rightarrow \mathbb{R}^{+}$ map bounded sets in precompact sets and are invariant.
\item There exists a constant $c>0$ such that $\mathcal{A}(x)\leq 0$ or $\norm{x}>c$ implies that $\gamma^+(x)=\gamma^{-}(x)=0$ and $K(x)=0$. 
%\item $h|_{E^{0}\oplus E^{-}}=\id_{E^{0}\oplus E^{-}}$. 
\end{itemize}

Let $S:=\{x\in E\, |\, \norm{x}_{E}=1\}$ and let $\xi\subseteq E$ be an invariant subset of $E$. We define the {\em pseudoindex} of $\xi$ as 
$$\ind(\xi):=\inf\big{\{}\alpha(h(\xi)\cap S\cap E^{+})\,\vert\, h\in\Gamma\big{\}}.$$
For the basic properties of the pseudoindex we refer to \cite{EH90}. In particular, the following result is often useful.
\begin{proposition}\cite[Proposition 1]{EH90}
If $V_{k}\subseteq E^{+}$ is a finite dimensional invariant subspace of $E^{+}$ of complex dimension $k$, then $\ind(V_k\oplus E^0\oplus E^-)=k$.
\end{proposition}

We need to introduce one last concept before defining the symplectic capacities. We call a smooth function $H\colon\C^n \rightarrow (0,+\infty)$ an {\em admissible Hamiltonian} for a bounded domain $\Omega\subset \C^n$ if \begin{itemize}
\item $H$ is $0$ on some open neighborhood of $\overline{\Omega}$. 
\item $H(z)=c|z|^2 $ for $|z|$ large enough, where $c>\pi$ and $c\notin \Z\pi$.
\end{itemize}
We denote by $\mathcal{H}(\Omega)$ the set of the admissible Hamiltonians for $\Omega$. 
For $j$ a positive integer and $H\in\mathcal{H}(\Omega)$, we define a number $c_{H,j}\in (0,+\infty)\cup\{\infty\}$ by
$$c_{H,j}:=\inf\big{\{}\sup \mathcal{A}_{H}(\xi)\,|\, \xi\subset E \text{ is $S^1$-invariant and }\ind(\xi)\geq j\big{\}}.$$
Here $\mathcal{A}_H\colon E\rightarrow\mathbb{R}$ is the action functional associated to a Hamiltonian $H$ defined by
$$\mathcal{A}_H(f):=\mathcal{A}(f)-\int_0^1H(f(t))\, dt.$$
Every number $c_{H,j}$ is non-negative and, if finite, is a critical value of $\mathcal{A}_H$ \cite[page 559]{EH90}. 

We can now define the $j$-th {\em Ekeland-Hofer symplectic capacity} of $\Omega$ as 
$$c_j(\Omega):=\inf_{H\in\mathcal{H}(\Omega)}c_{H,j}.$$

\begin{remark}\label{nuovoremmark} In our computation of the Ekeland-Hofer symplectic capacities, we will not use Hamiltonians with quadratic behavior at infinity, as in the definition. We will use instead Hamiltonians of the form $H(z)=f(r(z))$, where $r$ is the gauge function of the domain, and $f$ is linear at infinity (see \cite{BLMR85,F92}). This choice does not affect the values obtained in the computation of the capacities, as one can easily prove by combining the following two facts: 
\begin{itemize}
\item For Hamiltonians $H_1$ and $H_2$ we have  $$H_1\le H_2 \Longrightarrow \mathcal{A}_{H_1}\ge \mathcal{A}_{H_2}\Longrightarrow c_{H_1,k}\le c_{H_2,k}.$$
\item For every $f$ linear at infinity there exists $H\in \mathcal{H}(\Omega)$ such that $f(r)\le H$. Similarly, for any $H\in \mathcal{H}(\Omega)$ there exists $f$ such that $H\le f(r)$.  
\end{itemize} Observe also that if we choose $f$ so that the Hamiltonian $f(r)$ has no periodic solutions of period $1$ at infinity, then $\mathcal{A}_{f(r)}$ satisfies the Palais-Smale condition. This implies that $c_{f(r),k}$, if it is finite, is a critical value of $\mathcal{A}_{f(r)}$ (see \cite[page 71]{AB94}).
\end{remark} %fine nuovo

The next theorem from \cite{HZ90} characterizes the first symplectic capacity as the infimum of the action spectrum. 
\bt \label{c1smooth} Let $\Omega$ be a smoothly bounded convex domain in $\C^n$ and let $\alpha = \min \Sigma(\Omega)$. Then $c_1(\Omega)=\alpha$.
\et

%\begin{remark}From the proof of Theorem \ref{c1smooth} one can see see that, in general, if $W\subset E$ has index greater than $k$ and there exists $c>0$ such that 
%$$ \Psi_c (\zeta):=\mathcal{A}(\zeta)-c\int^1_0 r(\zeta(t))dt <0 \quad \forall \zeta\in W,$$
%then $c_k(\Omega)\le c$.
%\end{remark
Note that we cannot apply Theorem \ref{c1smooth} directly to the real bidisc, since its boundary $\partial (D^2\times D^2)$ is not smooth. In the next section we show how to overcome this difficulty by appropriately approximating $D^2\times D^2$ with smooth domains. 
\section{Approximation with smooth domains}\label{last}
We start by constructing a decreasing sequence of smooth convex domains $\mathcal{D}_n$ converging to $D^2\times D^2$. 
Let $g:\R\rightarrow [0,+\infty)$ be a convex, increasing function such that $g(1)=1$ and $g(s)=0$ for $s<0$. Consider the following subsets of $\C^2$:
\begin{equation}\label{4.1}
\mathcal{D}_n :=\{ z\in\C^2 |\ g(n(x^2_1+x^2_2 -1))+g(n(y^2_1+y^2_2 -1))\le 1\} .
\end{equation}
The domains defined in \eqref{4.1} are smooth and convex. Moreover, they satisfy the following properties: 
\begin{itemize}
\item $\mathcal{D}_n \supset \mathcal{D}_{n+1}$ for all $n\in\mathbb{N}$. 
\item $\bigcap^{+\infty}_{n=1} \mathcal{D}_n=D^2\times D^2$.
\item For all $n\in\mathbb{N}$, we have $$\frac 1 {\sqrt{1+\frac 1n}}\mathcal{D}_n \subset D^2\times D^2.$$
\end{itemize}
In particular, by the properties of the capacities, for any choice of positive integers $k$ and $n$, the following double inequality holds:
$$ \frac 1 {1+\frac 1n} c_k(\mathcal{D}_n)\le c_k(D^2\times D^2)\le c_k(\mathcal{D}_n) .$$  

The next proposition shows how the closed characteristics of the real bidisc $D^2\times D^2$ are approximated by the closed characteristics of the approximating domains $\mathcal{D}_n$. Before stating the result, it is convenient to give the following definition. 
\begin{definition} For all $M>0$ and $\varepsilon>0$ we define 
$$\Sigma_{M,\varepsilon}(\mathcal{D}_n):= (\Sigma(\mathcal{D}_n)\cap[0,M])\setminus \bigcup^{\infty}_{k=1}\,[2k\pi-\varepsilon,2k\pi+\varepsilon].$$
Here $\Sigma(\mathcal{D}_n)$ denotes the action spectrum of $\mathcal{D}_n$.
\end{definition}  
\begin{proposition} \label{4.0}
Let $\alpha\in \Sigma(D^2\times D^2)$ and suppose that $\alpha=\mathcal{A}(\gamma)$, where $\gamma$ is a non-gliding closed characteristic of $D^2\times D^2$. Then there exists a sequence $\gamma_n$, where each $\gamma_n$ is a closed characteristic of $\mathcal{D}_n$, such that $\gamma_n$ converges to $\gamma$ and $\mathcal{A}(\gamma_n)$ converges to $\alpha$. In particular, for all $M>0$ and $\varepsilon>0$, we have
$$ d(\Sigma_{M,\varepsilon}(\mathcal{D}_n),\Sigma_{M,\varepsilon}(D^2\times D^2)) \to 0\text{ for }n\to +\infty, $$
where $d(\cdot,\cdot)$ is the Hausdorff distance between sets.
\end{proposition}
\begin{proof}
Note that, for every $n$, we can decompose the boundary $\partial \mathcal{D}_n$ of $\mathcal{D}_n$ into three components
\begin{equation}\label{4.2}\begin{split}
X^3_n:=\left\{ (z_1,z_2)\in\C^2, \ x^2_1+x^2_2 -1< 0,  y^2_1+y^2_2= 1+\frac 1n \right\} \\ 
Y^3_n:=\left\{ (z_1,z_2)\in\C^2, \ y^2_1+y^2_2 -1< 0,  x^2_1+x^2_2= 1+\frac 1n \right\}
\end{split}
\end{equation} 		
 and
\begin{equation}\label{4.3}
T^3_n:=\partial\mathcal{D}_n\setminus (X^3_n\cup Y^3_n).
\end{equation} 				
To find the closed characteristics of $\mathcal{D}_n$, we consider the system of differential equations
\begin{equation}\label{4.4}
\begin{cases}
\dot{x} =-g'(n(|y|^2 -1))2ny \\
\dot{y} =g'(n(|x|^2 -1))2nx,
\end{cases} 
\end{equation}
where $x=(x_1,x_2), y=(y_1,y_2)$. We first note that $\det({x,y})=x_1y_2-x_2y_1$ is constant along the solutions of \eqref{4.4}. It is convenient to use polar coordinates in the planes defined by the variables $x$ and $y$ respectively. We recall the notation already introduced in the proof of Proposition \ref{s1}:
\begin{equation*}\begin{split}
r_1e^{i\varphi_1}:=r_1(\cos\varphi_1,\sin\varphi_1)=(x_1,x_2)\\
r_2e^{i\varphi_2}:=r_2(\cos\varphi_2,\sin\varphi_2)=(y_1,y_2).
\end{split}
\end{equation*}
%We make the obvious remark that the complex structure described in \eqref{4.5} is not related to the complex structure of $\C^2$. 
We can now rewrite the system \eqref{4.4} as
\begin{equation}\label{4.6}
\begin{cases}
(\dot{r_1} +ir_1\dot{\varphi_1})e^{i\varphi_1}=-g'(n(r^2_2 -1))2nr_2e^{i\varphi_2} \\
(\dot{r_2} +ir_2\dot{\varphi_2})e^{i\varphi_2}=g'(n(r^2_1 -1))2nr_1e^{i\varphi_1},
\end{cases} 
\end{equation}
from which we obtain the two systems
\begin{equation}\label{4.7}
\begin{cases}
r_1\dot{\varphi_1}=g'(n(r^2_2-1))2nr_2\sin(\varphi_1 -\varphi_2) \\
r_2\dot{\varphi_2}=g'(n(r^2_1-1))2nr_2\sin(\varphi_1 -\varphi_2) \\
r_1r_2\sin(\varphi_1-\varphi_2)=\det(x,y)=\const
\end{cases}
\end{equation}
and 
\begin{equation}\label{4.8}
\begin{cases}
r_1\dot{r_1}=-g'(n(r^2_2 -1))2nr_3 \\
r_2\dot{r_2}=g'(n(r^2_1 -1)2nr_3 \\
\dot{r_3}=2n(g'(n(r_1^2-1))r^2_1-g'(n(r^2_2-1))r^2_2).
\end{cases}
\end{equation}
Here $r_3:=x\cdot y=x_1y_1+x_2y_2=r_1r_2\cos(\varphi_1-\varphi_2)$.

Clearly the characteristics in $\mathcal{D}_n$ are straight segments when they lie on $X^3_n$ or $Y^3_n$. We now want to understand the  behavior of the characteristics at their intersections with $T^3_n$.
Let us consider the Cauchy problem \eqref{4.4} with initial data $x(0)=(1,0)$ and $y(0)=\sqrt{1+\frac 1n}(\cos\theta_0,\sin\theta_0)$ for $\theta_0\in (0,\pi/2)$. The corresponding solution of \eqref{4.4} enters $X^3_n$ 
(this can be seen by inspection of \eqref{4.8}, since $r_3(0)>0$ implies that $r_1$ is decreasing and $r_2$ is constant). Reasoning as in Proposition \ref{s1}, we see that the solution reaches the point 
$$(x_1,x_2)=\big{(}\cos(\pi +2\theta_0),\sin(\pi+2\theta_0)\big{)},\quad (y_1,y_2)=\sqrt{1+\frac 1n}(\cos\theta_0,\sin\theta_0).$$ The solution then enters $T^3_n$ at a time $t_0$. 
%Since we are dealing with autonomous systems, we can assume without loss of generality that $t_0=0$. 
%We want to show that the solution leaves $T^3_n$ 
%\textcolor{red}{approximately} at the point 
%\begin{equation*}
%(x_1,x_2)=\sqrt{1+\frac 1n}(\cos(\pi +2\theta_0),\sin(\pi+2\theta_0)),\quad (y_1,y_2)= (\cos\theta_0,\sin\theta_0).
%\end{equation*}
From \eqref{4.8} we see that $r_3(t_0)<0$, $r_1$ increases and $r_2$ decreases. Let $T>0$ be the smallest positive real number such that $\dot{r_3}(t_0+T)=0$. By simmetry we see that, $r_1(t)=r_2(2T+2t_0-t)$, $r_2(t)=r_1(2T+2t_0-t)$ and $r_3(t)=r_3(2T+2t_0-t)$. In particular, this tells us that the solution eventually leaves $T^3_n$. By the third equation in \eqref{4.7}, the angle between $x(t_0+2T)$ and $y(t_0+2T)$ is the same as the 
angle between $x(t_0)$ and $y(t_0)$ but $\varphi_1(t_0)\neq \varphi_1(t_0+2T)$. We want to compute $\Delta\varphi:=\varphi_1 (t_0+2T)-\varphi_1 (t_0)=\varphi_2 (t_0+2T)-\varphi_2 (t_0)$. By \eqref{4.7}
\begin{equation}\label{4.9}
\begin{split}
\Delta\varphi&=\int^{2T}_0 \frac{g'(n(r^2_2-1))2nr_2}{r_1}2\sin(\varphi_1-\varphi_2)dt \\ 
&=\int^{2T}_0 \frac{g'(n(r^2_2-1))2n}{r^2_1} \sqrt{1+\frac 1n}\sin\theta_0dt. 
\end{split}  
\end{equation}
Since $r_3 =x\cdot y =r_1r_2 \cos(\varphi_1 -\varphi_2)$, equation \eqref{4.8} implies
\begin{equation}\label{4.10}
\begin{split}
 \dot{(r_1^2)}&=-g'(n(r^2_2-1))2n r_1r_2\cos(\varphi_1-\varphi_2)\\ &=-g'(n(r^2_2 -1))2n \sqrt{r^2_1r^2_2 -\Big{(}1+\frac 1n\Big{)}\sin^2\theta_0}.
\end{split}
\end{equation}
Using \eqref{4.10} inside \eqref{4.9} we obtain
\begin{equation}\label{4.11}
\Delta\varphi=-\int^{1+\frac 1n}_1 \frac{\sqrt{1+\frac 1n}\sin(\theta_0)}{r^2_1 \sqrt{r^2_1r^2_2 -\big{(}1+\frac 1n\big{)}\sin^2(\theta_0)}}d(r^2_1).
\end{equation}
It follows from $$g(n(r^2_1-1))+g(n(r^2_2-1))=1$$ that 
\begin{equation}\label{questa}
r^2_2 =1+\frac{g^{-1}(1-g(n(r^2_1-1)))}n. 
\end{equation}
Plugging \eqref{questa} into \eqref{4.11} we obtain
\begin{equation}\label{la12}
\Delta\varphi =-\int^{1+\frac 1n}_1 \frac{\sqrt{1+\frac 1n}\sin\theta_0}{u \sqrt{u(1+\frac{g^{-1}(1-g(n(u-1)))}n) -(1+\frac 1n)\sin^2\theta_0}}du.
\end{equation}
From \eqref{la12} we see that $\Delta\varphi$ is small for $\theta_0 <\frac \pi 2$ and $n$ large.
Following the same reasoning as in Proposition \ref{s1} we can see that after $2m$ straight sides the characteristic hits the point
$$ P_{2m}=\bigg{(} (-1)^m e^{i(2m(\theta +\Delta\varphi))},(-1)^m\sqrt{1+\frac 1n}e^{i((2m+1)\theta +2m\Delta\varphi)}\bigg{)}.$$
The characteristic is closed if and only if 
\begin{equation}\label{cc1}
\begin{cases}
\theta+\Delta\varphi=\frac{k\pi}{m} \text{ for } k=0,...,\frac m2 -1 \text{ if $m$ is even} \\
\theta+\Delta\varphi=\frac{2k-1}{2m}\pi \text{ for } k=1,...,\frac{m-1}2\text{ if $m$ is odd.} 
\end{cases}
\end{equation}
Since $\Delta\varphi$ depends continuously on $\theta_0$ and is small if $n$ is large, then the equations in \eqref{cc1} are solvable.
In particular, if $\theta_{k,m}\in J_m$ (see Proposition \ref{s1}), for $n$ big enough there exists $\theta$ close to $\theta_{k,m}$ such that \eqref{cc1} is satisfied. The corresponding characteristics of $\mathcal{D}_n$ and $D^2\times D^2$ 
are close to each other and their actions are also close.
Note that there could also be some characteristics that are entirely contained in $T^3_n$. These characteristics are left out by the description above. If $n$ is large, they are close to the gliding trajectories of the bidisc $D^2\times D^2$ and their actions are close to a multiple of $2\pi$. 
\end{proof}

\section{The Ekeland-Hofer symplectic capacities of the real bidisc}\label{SectionCap}
In this section we compute the symplectic capacities of the real bidisc $D^2\times D^2$. The following result will be the main tool for our computations.
\begin{proposition}\label{ckfinite} Let $\Omega$ be a smooth convex domain in $\C^n$ containing $0$, and let $r$ be its gauge function. For all $c>0$, let $\Psi_c:E\rightarrow \R$ be the functional
\begin{equation}\label{Psi}
 \Psi_c(\zeta):=\mathcal{A}(\zeta)-c\int^1_0 r(\zeta(t))\, dt .
 \end{equation}
If $W\subset E$ is an invariant subset of pseudoindex at least $k$ such that $\Psi_c|_W \le 0$, then $c_k(\Omega)\le c$.
%\begin{equation}\label{Wb}
%\Psi|_W \le 0
%\end{equation}

\end{proposition}
\begin{proof} Let $\epsilon >0$ and choose a smooth function $f_\epsilon :[0,\infty)\rightarrow [0,\infty) $ such that
\begin{equation}
f_\epsilon (s) =\begin{cases} 0 \,\,\,\,\,\,\,\,\,\,\,\,\,\,\,\,\,\,\,\text{ if $s\le 1$} \\
(c+\epsilon)s \,\,\text{ if $s$ is large.}
\end{cases}
\end{equation}
Moreover, we require that $0\le f_{\epsilon}'(s)\le c+2\epsilon$ for all $s$, and that $f_{\epsilon}'(s) \in \Sigma(\Omega)$ only for finitely many numbers $s_1,\dots,s_m$, which we can assume to be arbitrarily close to $1$.
Let $f_{\epsilon}'(s_j)=\alpha_j\in\Sigma(\Omega)$ and define $H_{f_\epsilon}(z):= f_\epsilon(r(z))$. The periodic orbits of $\dot{z}=J\nabla H_{f_\epsilon}(z)$ of period $1$ are obtained by scaling. Namely, they are the curves $\sqrt{s_j}\gamma_j(\alpha_j t)$, where $\gamma_j $ is the closed characteristic in $\partial\Omega$ such that $\mathcal{A}(\gamma_j)=\alpha_j$. The corresponding critical values of $\mathcal{A}_{H_{f_\epsilon}}$ are 
\begin{equation}\label{cv}
\mathcal{A}_{H_{f_\epsilon}}(\sqrt{s_j}\gamma_{j}(\alpha_j t))=s_j\alpha_j -f_{\epsilon}(s_j).
\end{equation}
As $\epsilon \to 0$, these critical values tend to $\alpha_j$, and each of them is less than $c$.  Note that
\begin{equation*} 
\begin{split}
\mathcal{A}_{H_{f_\epsilon}}(\zeta)&=\Psi_c(\zeta) +\int^1_0 \big{[}cr(\zeta(t))-f_{\epsilon}(r(\zeta(t)))\big{]}dt\\
&\le \Psi_c (\zeta) + \int^1_0 \big{[}C-\epsilon r(\zeta(t))\big{]}dt\le \Psi_c (\zeta) +D 
\end{split}
\end{equation*}
for some constants $C$ and $D$. Recalling that $\Psi|_W \le 0$, then
\begin{equation}\label{this}
\mathcal{A}_{H_{f_\epsilon}}|_W \le C
\end{equation}
for some new constant $C$. Equation \eqref{this} implies $c_{{H_{f_\epsilon}},k} <\infty$. Since there are no periodic orbits at infinity of period $1$, then $\mathcal{A}_{H_{f_\epsilon}}$ satisfies the Palais-Smale condition and therefore $c_{{H_{f_\epsilon}},k}$ is a critical value of $\mathcal{A}_{H_{f_\epsilon}}$. Hence $c_{{H_{f_\epsilon}},k}\le c$ by \eqref{cv}.
Since $c_k(\Omega)\le c_{{H_{f_\epsilon}},k}$, the conclusion follows.
\end{proof}

In the next theorem we compute the first Ekeland-Hofer capacity of the real bidisc. A different proof of the same result appears in \cite{AO14}.

\bt \label{c1} For the unit real bidisc $D^2\times D^2$ we have $c_1(D^2\times D^2) =4$. \et
\begin{proof} 
Recall that $4=\min\Sigma(D^2\times D^2)$ (Remark \ref{rmk}). Theorem \ref{4.0} then implies the existence of a sequence of closed characteristics $\gamma_n\subset \partial \mathcal{D}_n$ with $\alpha_n:=\mathcal{A}(\gamma_n)\to 4$. By Theorem \ref{c1smooth} we have 
\begin{equation}\label{capa}
\lim_{n\to +\infty} c_1(\mathcal{D}_n)= 4 .
\end{equation}
Now observe, from the construction of the $\mathcal{D}_n$, that $$\frac{1}{\sqrt{1+\frac 1n}}\mathcal{D}_n\subset D^2\times D^2\subset  \mathcal{D}_n.$$ The basic properties of the capacities then yield the double inequality 
$$ \frac 1{1+\frac 1n}c_1(\mathcal{D}_n)\le  c_1(D^2\times D^2)\le c_1(\mathcal{D}_n),$$
from which we can conclude that $c_1(D^2\times D^2) =4$ by applying \eqref{capa}.
 \end{proof}
For the next theorem we need the following simple lemma.
\bl \label{stima} Let $f(t)=\sum_{k\in\Z} f_k e^{2k\pi i t }$ be an $L^1$-convergent series. Then, for every integer $n$,
\begin{equation*}
\int^1_0 |f(t)| dt \ge  |f_n|. 
\end{equation*}
\el
\begin{proof} For every $n$ we have 
\begin{equation*}
\begin{split}
\int^1_0 \big{|} \sum_{k\in\Z} f_k e^{2k\pi i t }\big{|}dt &= \int^1_0 \big{|}e^{2n\pi i t} \sum_{k\in\Z} f_k e^{2(k-n)\pi i t }\big{|}dt \\& \ge \left| \int^1_0  \sum_{k\in\Z} f_{k+n} e^{2k\pi i t } dt \right| =|f_n|.
\end{split}
\end{equation*}
\end{proof}
%Nuova dimostrazione per c_2
\bt\label{c2bis} For the real bidisc $D^2\times D^2$ we have $c_2(D^2\times D^2)=3\sqrt{3}$. \et
\begin{proof}

Let $W$ be the following invariant subset of $E$:
$$ W:=\bigg{\{}(\alpha,\beta) e^{2\pi i t} +\gamma([\alpha:\beta])\Big{(}\overline{\alpha} \frac{\alpha^3}{|\alpha|^3},\overline{\beta} \frac{\alpha^3}{|\alpha|^3}\Big{)}e^{4\pi i t}\,\Big{\vert}\,\alpha,\beta\in \C \bigg{\}}\oplus E^0\oplus E^-, $$
where $\gamma:\P^1_\C \rightarrow \R$ is a non-negative continuous function which is non-zero only in a neighborhood of the two points $[1:i]$ and $[1:-i]$.  We will specify later how $\gamma$ is chosen.
We now prove %following \cite{EH90} 
that $W$ has pseudoindex $2$.  
First note that 
\begin{equation*}
\begin{split}
W\cap S \cap E^+=\bigg{\{}(\alpha,\beta) e^{2\pi i t} +\gamma([\alpha:\beta])\Big{(}\overline{\alpha} \frac{\alpha^3}{|\alpha|^3},\overline{\beta} 
\frac{\alpha^3}{|\alpha|^3}\Big{)}&e^{4\pi i t}\,\Big{\vert}\,\alpha,\beta\in \C, \\ &(|\alpha|^2+|\beta|^2)(1+\gamma^2)=1\bigg{\}}, 
\end{split}
\end{equation*}
and therefore $W\cap S \cap E^+$ has index $2$. Assume now by contradiction that there exists $h\in \Gamma$ such that $F:=h(W)\cap E^+ \cap S$ has index strictly smaller than $2$.
Then, by the properties of the index, there exists an open neighborhood of $U$ of $F$ in $E$ such that $\alpha(U)=\alpha(F)$. Let $E_k=\{f\in E|\ f_j=0 \text{ for }|j|>k\}$ and 
denote by $Q_k:E\rightarrow E_k$ the corresponding orthogonal projection. We claim, for $k$ large, that
\begin{equation}\label{27}
 (Q_k (h(W\cap E_k)))\cap S\cap E^+ \subset U.
\end{equation} 
Assume by contradiction that \eqref{27} is false. Then there exists a sequence of functions $f_k \in E_k \cap W$
such that $Q_k(h(f_k))\in E^+\cap S$ and $ Q_k(h(f_k))\notin U$.
Recalling the structure of the homeomorphism $h$ (see \eqref{groupGamma}), we have
\begin{equation}\label{indW1}
\begin{cases}
f^0_k+e^{\gamma^- (f_k)}f^-_k +(P^-+P^0)Q_kK(f_k)=0 \\
\norm{e^{\gamma^+(f_k)}f_k^+ +P^+Q_kK(f_k)} =1.
\end{cases}
\end{equation}
Note that $f_k$ must be a bounded sequence, otherwise we have $K(f_k)=0$, which together with \eqref{indW1} implies $\|f^+_k\|=1$
and $f^0_k=f^-_k=0$, thus giving a contradiction.
We can therefore assume that $K(f_k)$ and $\gamma^{\pm}(f_k)$ converge. Hence, by \eqref{indW1}, the sequences $f^-_k$ and $f^0_k$ also converge.
Furthermore, the sequence $f^+_k$ converges as well, since it lies in a finite dimensional space. We therefore have that $f_k$ converges to some element $f_{\infty}\in W$ with $h(f_{\infty})\in E^+\cap S$ and $h(f_\infty)\notin U$, which is a contradiction.

In order to apply \cite[Proposition 3.3]{FHR82} as done in \cite[page 558]{EH90}, we consider the following ``truncated" set: for $M>0$ let
\begin{equation*}
\begin{split}
 W_M:=\bigg{\{}(\alpha,\beta) e^{2\pi i t} &+ \chi\bigg{(}\frac{|\alpha|^2+|\beta|^2+\|f^0+f^-\|^2}M\bigg{)}\gamma([\alpha:\beta])\Big{(}\overline{\alpha} \frac{\alpha^3}{|\alpha|^3},\overline{\beta} \frac{\alpha^3}{|\alpha|^3}\Big{)}e^{4\pi i t}\\&+f^0+f^-\,\Big{\vert}\,\alpha,\beta\in \C, f^0\in E^0, f^-\in E^-\bigg{\}},
 \end{split}
\end{equation*}
where $\chi$ is a smooth function such that $\chi(t)=1$ for $t<1$ and $\chi(t)=0$ for $t>\frac32$.
Note that $W_M$ coincides with $W$ inside the ball of radius $M$ in $E$ and that $h(W_M)\cap E^+\cap S=h(W)\cap E^+\cap S$.
Consider now the equivariant map
\begin{equation*}
\phi:\C^2\oplus E^0 \oplus (E^-\cap E_k) \rightarrow W_M
\end{equation*}
\begin{equation}
\begin{split}
\phi(\alpha,\beta,f^0,f^-_k)=&\chi\bigg{(}\frac{|\alpha|^2+|\beta|^2+\|f^0+f^-\|^2}M\bigg{)}\gamma([\alpha:\beta])\Big{(}\overline{\alpha} \frac{\alpha^3}{|\alpha|^3},\overline{\beta} \frac{\alpha^3}{|\alpha|^3}\Big{)}e^{4\pi i t}\\&+(\alpha,\beta) e^{2\pi i t}+f^0+f^-_k
\end{split}
\end{equation}
By \cite[Proposition 3.3]{FHR82} applied to the map $Q_k h(\phi)$ we obtain $$\alpha(Q_k(W_M)\cap E^+\cap S) \ge 2.$$ The conclusion that $\ind(W)=2$ is achieved by taking $M$ large enough.
%fine dimostrazione dell'indice =2

Let now $r$ be the gauge function of $D^2\times D^2$ and $\Psi_c:E\rightarrow \R$ the functional defined in \eqref{Psi}:
$$ \Psi_c(\zeta):=\mathcal{A}(\zeta)-c\int^1_0 r(\zeta(t))\, dt .$$
We recall that
$$r(z_1,z_2)=\frac{|z_1|^2 +|z_2|^2}2 +\frac{|\Re (z_1^2+z_2^2)|}2.$$
We will prove that $\Psi_c|_{W} <0$ for $c=4\sqrt{2}$.

Let $v=(\zeta_1,\zeta_2)\in W$. Then
$$v=v_1e^{2\pi i t}+v_2e^{4\pi i t}+\sum^{+\infty}_{k=0} w_{-k}e^{-2k\pi i t},$$
where $v_1=(\alpha,\beta)$ and $v_2=\gamma \frac{\alpha^3}{|\alpha|^3}(\overline{\alpha},\overline{\beta})$ for some $\alpha,\beta\in\C.$
We have
\begin{equation}\label{c21}
\begin{split}
\Psi_c(v)=&\Big{(}\pi -\frac{c}{2}\Big{)}|v_1|^2+\Big{(}2\pi -\frac c2\Big{)}|v_2|^2 -\sum_{k=1}^{+\infty}k\pi |w_{-k}|^2 \\&-\frac c2 \sum_{k=0}^{+\infty}|w_{-k}|^2-\frac c2 \int^1_0 \Big{|}\Re\big{(}\zeta_1^2(t)+\zeta_2^2(t)\big{)}\Big{|}dt .\\
\end{split}
\end{equation}
To estimate the integral on the right side of \eqref{c21} we compute the Fourier coefficients of order 4 and 6 of $\Re\big{(}\zeta_1^2(t)+\zeta_2^2(t)\big{)}$. We denote them respectively by $I_4$ and $I_6$. 
\begin{equation}\label{c22}
2I_4 =v_1^2 +2(v_2\cdot w_0)+\overline{(w_{-1})^2 +2(w_{-2}\cdot w_{0})+2(w_{-3}\cdot v_{1})+2(w_{-4}\cdot v_{2})},
\end{equation}
\begin{equation}\label{c23}
2I_6=2(v_{2}\cdot v_{1})+\overline{2(w_{-1}\cdot w_{-2})+2(w_{0}\cdot w_{-3})+2(v_{1}\cdot w_{-4})+2(v_{2}\cdot w_{-5})}.
\end{equation}
By Lemma \ref{stima} applied to $I_6$ and equation \eqref{c23} we get
\begin{equation}\label{I6}
 \int^1_0 \Big{|}\Re\big{(}\zeta_1^2(t)+\zeta_2^2(t)\big{)}\Big{|}dt \ge |I_6|\ge |v_2\cdot v_1|-|w_{-1}\cdot w_{-2}|-|w_{0}\cdot w_{-3}|-|v_{1}\cdot w_{-4}|-|v_{2}\cdot w_{-5}|.
 \end{equation}
Since all the functionals on the right side of \eqref{I6}, which we are going to estimate, are homogeneous, it is not restrictive to assume that $|\alpha|^2+|\beta|^2=1$. In particular, $|v_2\cdot v_1|=|\gamma|$. Applying the Cauchy-Schwarz inequality we obtain the estimates
\begin{equation}\label{CSS}
|w_{-1}\cdot w_{-2}|\le \frac 12 \big{(}|w_{-1}|^2+|w_{-2}|^2\big{)} \,\,\text{ and }\,\, |w_{0}\cdot w_{-3}|\le \frac 12 \big{(}|w_{0}|^2+|w_{-3}|^2\big{)}.
\end{equation} 
For the terms $|v_{1}\cdot w_{-4}|$ and $|v_{2}\cdot w_{-5}|$ we give small and large constant to obtain
\begin{equation} \label{CSSP}
\begin{split}
2|v_{1}\cdot w_{-4}|&\le \frac{\frac c2 +4\pi}{\frac  c4}|w_{-4}|^2 +\frac{\frac c4}{\frac c2 +4\pi} \\
2|v_{2}\cdot w_{-5}|&\le \frac{\frac c2 +5\pi}{\frac  c4}|w_{-5}|^2 +\frac{\frac c4}{\frac c2 +5\pi}\gamma^2,
\end{split}
\end{equation}
where we have used that $|v_1|=1$ and $|v_2|^2=\gamma^2$.
Combining the estimate for the integral in \eqref{I6} with \eqref{CSS} and \eqref{CSSP} and plugging into \eqref{c21}, we see that
\begin{equation}\label{neg}
\Psi_c(v)<\left(\pi-\frac c2 -\frac c2 \gamma +\frac{\left(\frac c4\right)^2}{\frac c2 +4\pi}\right) +\left(2\pi-\frac c2+\frac{\left(\frac c4\right)^2}{\frac c2 +5\pi}\right)\gamma^2+\dots,
\end{equation}
where dots stand for negative terms that arise after absorbing the terms that involve $w_{-4}$ and $w_{-5}$ on the right side of \eqref{CSSP} with the terms $(-4\pi -\frac c2)|w_{-4}|^2$ and 
$(-5\pi-\frac c2)|w_{-5}|^2$ on the right side of \eqref{c21}. The right side of \eqref{neg} is negative if 
\begin{equation}\label{inequa}
\left(2\pi-\frac c2+\frac{\left(\frac c4\right)^2}{\frac c2 +5\pi}\right)\gamma^2 -\frac c2 \gamma +\left(\pi-\frac c2 +\frac{\left(\frac c4\right)^2}{\frac c2 +4\pi}\right) <0.
\end{equation}
For $c=4\sqrt{2}$ the inequality in \eqref{inequa} is satisfied by the solutions to
$$ 0.44-2.82\gamma+3.56\gamma^2 <0.$$
In particular, $\Psi_c(v)<0$ for 
\begin{equation}\label{rangegamma}
0.22\le \gamma\le 0.58 .
\end{equation}
Using $I_4$ in place of $I_6$, we get that
\begin{equation}\label{I_4}
\begin{split}
 \int^1_0 \Big{|}\Re\big{(}\zeta_1^2(t)+\zeta_2^2(t)\big{)}\Big{|}dt\ge |I_4|\ge &|\alpha^2+\beta^2 |-|v_{2}\cdot w_{0}|-|w_{0}\cdot w_{-2}|\\&-|v_{1}\cdot w_{-3}|-|v_{2}\cdot w_{-4}|-\frac{|w_{-1}|^2}2.\\
\end{split}
\end{equation}

We first estimate
\begin{equation}\label{estimate}
\begin{split}
2|v_{2}\cdot w_{0}|&\le \frac{\frac 34 c +4\pi}{\frac c2 +2\pi}|w_0|^2+\frac{\frac c2 +2\pi}{\frac 34 c +4\pi}\,\gamma^2,\\ 
2|w_{0}\cdot w_{-2}|&\le \frac{\frac c2 +2\pi}{\frac c4} |w_{-2}|^2 +\frac{\frac c4}{\frac c2 +2\pi}\,|w_0|^2, \\
2|v_{1}\cdot w_{-3}|&\le \frac{\frac c2 +3\pi}{\frac c4}|w_{-3}|^2+\frac{\frac c4}{\frac c2 +3\pi}, \\
2|v_{2}\cdot w_{-4}|&\le \frac{\frac c2+4\pi}{\frac c4}|w_{-4}|^2+\frac{\frac c4}{\frac c2 +4\pi}\,\gamma^2. 
\end{split}
\end{equation}
We have used again that $|v_1|^2=1$ and $|v_2|^2=\gamma^2$. Combining \eqref{estimate} with \eqref{I_4} and replacing inside \eqref{c21} we obtain
\begin{equation*}
\begin{split}
\Psi_c(v)\le &\left( \pi-\frac c2 -\frac c4|\alpha^2+\beta^2| +\frac{\left(\frac c4 \right)^2}{\frac c2 +3\pi}\right) \\ &+\left(2\pi-\frac c2 +\frac c4 \left(\frac{\frac c2 +2\pi}{\frac 34c+4\pi} 
+\frac{\frac c4}{\frac c2 +4\pi}\right)\right)\gamma^2 +\dots
\end{split}
\end{equation*}
Once again, dots indicate negative terms which come after absorption with the negative terms involving $|w_k|^2$ for $k=0,-2,-3,-4$.
Hence $\Psi_c(v)<0$ if 
\begin{equation}\label{ineq}
\left( \pi-\frac c2 +\frac{\left(\frac c4 \right)^2}{\frac c2 +3\pi}\right)-\frac c4|\alpha^2+\beta^2|  +\left(2\pi-\frac c2 +\frac c4 \left(\frac{\frac c2 +2\pi}{\frac 34c+4\pi}\right)+\frac{\frac c4}{\frac c2 +4\pi}\right)\gamma^2<0.
\end{equation}
For $c=4\sqrt{2}$, the inequality \eqref{ineq} is satisfied by the solutions to
\begin{equation}\label{ciao}
 0.477-1.41|\alpha^2+\beta^2|+4.352\gamma^2<0 .
\end{equation} Letting $\gamma_0=0.23$ by \refeq{rangegamma} and solving the corresponding equation to \eqref{ciao}, we can define $\delta_0:=\frac{0.47+4.22\cdot0.23^2}{1.41}=0.496$.
We then choose a continuous function $\gamma$ such that $0\le\gamma\le\gamma_0$ and
\begin{equation}
\gamma([\alpha:\beta])=
\begin{cases}
\gamma_0 \,\,\text{ if }|\alpha^2+\beta^2|<\delta_0=0.502 \\
0 \,\,\,\,\,\,\text{ if }|\alpha^2+\beta^2|>0.6.
\end{cases}
\end{equation}
With this choice, we conclude that $\Psi_{c}|_{W}<0$. The conclusion remain valid even if we choose a constant $c$ slightly smaller than $4\sqrt{2}$.

Now let $\mathcal{D}_n$ be the approximating sets for $D^2\times D^2$ constructed in Proposition \ref{4.0}. Note that 
\begin{equation}\label{inclusions}
\frac{\mathcal{D}_n}{\sqrt{1+\frac 1n}} \subseteq D^2\times D^2 \subseteq  \mathcal{D}_n.
\end{equation}
The first inclusion in \eqref{inclusions} implies 
\begin{equation}\label{Ger}
\frac{r_n}{1+\frac{1}{n}}\geq r,
\end{equation}
where $r_n$ is the gauge function of the set $\mathcal{D}_n$. For each $n$, let $\Psi_c^n$ be defined as

$$\Psi^n_c (\zeta):= \mathcal{A}(\zeta)-c\int^1_0 \frac{r_n(\zeta(t))}{1+\frac{1}{n}}dt, \,\quad \zeta\in E.$$
By \eqref{Ger} we have $\Psi^n_c\leq \Psi_c\leq 0$. Proposition \ref{ckfinite} and Proposition \ref{4.0} then imply
$$\lim_{n\to +\infty} \frac{c_2(\mathcal{D}_n)}{1+\frac{1}{n}}=3\sqrt{3}.$$ 
This is because $\Sigma(\mathcal{D}_n \cap [0,4\sqrt{2}-\varepsilon] =\{\alpha_1^n,\alpha_2^n\}$ for $\varepsilon>0$ and $n$ large, where $\alpha_1^n$ and $\alpha_2^n$ are two sequences
converging to $4$ and $3\sqrt{3}$ respectively (Proposition \ref{4.0}). Together with the fact that $\frac{c_2(\mathcal{D}_n)}{1+\frac 1n}<4\sqrt{2}$  we have that either $c_2(\mathcal{D}_n)=\alpha_1^n$ or $c_2(\mathcal{D}_n)=\alpha_2^n$. 
It cannot be $c_2(\mathcal{D}_n)=\alpha_1^n=c_1(\mathcal{D}_n)$, otherwise the set of characteristics of $\mathcal{D}_n$ of action $\alpha^n_1$ 
would have index $2$, and this is not the case.
Hence $c_2(\mathcal{D}_n)=\alpha^n_2 \to 3\sqrt{3}$, and this gives the conclusion.
\end{proof}

\begin{remark}\label{rmkRamos}
One can deduce that $c_2(D^2\times D^2)=3\sqrt{3}$ from the work of Ramos \cite{R17}. Let $E(a,b)$ denote the ellipsoid 
\begin{equation}\label{ellipsoid}
E(a,b)=\bigg{\{} (z_1,z_2)\in\C^2\,\Big{|} \pi\Big(\frac{|z_1|^2}{a}+\frac{|z_2|^2}{b}\Big)\leq 1 \bigg{\}}.
\end{equation} 
It follows from \cite[Corollary 9]{R17} that $E(4-\varepsilon,4+\delta)$ embeds symplectically into the bidisc $D^2\times D^2$ for some $\varepsilon,\delta>0$. Hence \[4+\delta\leq c_2(D^2\times D^2)\leq 3\sqrt{3}.\] Since $3\sqrt{3}$ is the only number in the spectrum $\Sigma(D^2\times D^2)$ with this property, we conclude that $c_2(D^2\times D^2)=3\sqrt{3}$.
\end{remark}

\bt \label{cap3}
For the unit real bidisc $D^2\times D^2$ we have $c_3(D^2\times D^2) =8$.
\et 
 \begin{proof}
Let $W$ be the subspace of $E$ defined by
$$W := E^-\oplus E^0\oplus\langle (e^{2\pi i t},0),(0,e^{2\pi i t}), (e^{4\pi i t},0)\rangle.
$$

By \cite[Proposition 1]{EH90} the pseudoindex of $W$ is equal to 3. We now prove that for some constant $c$ we have $\Psi_c|_W \le 0$, where $\Psi_c$ is the functional defined in \eqref{Psi}. 
For an element $(\zeta_1,\zeta_2)=(\alpha e^{2\pi i t}+\gamma e^{4\pi i t},\beta e^{2\pi i t}) +w^-+w^0\in W$, with $\alpha,\beta,\gamma\in \C$, $ w^-\in E^-, w^0\in E^0$, we have
\begin{equation}\label{lastint}
\begin{split}
&\Psi_c((\alpha e^{2\pi i t}+\gamma e^{4\pi i t},\beta e^{2\pi i t}) +w^-+w^0)=\mathcal{A}(\zeta) -c\int^1_0 r(\zeta(t))dt \\
&=- \frac c2 ( |\alpha|^2 +|\beta|^2 +|\gamma|^2) -\| w^-\|_E -\frac c2 \|w^0 +w^-\|^2_{L^2}-\frac c2 \int^1_0 \big{|}\Re (\zeta_1^2(t) +\zeta^2_2(t))\big{|} dt \\&\,\,\,\,\,\,\,\,+\pi( |\alpha|^2 +|\beta|^2 +2|\gamma|^2).
\end{split}
\end{equation}
To give an estimate of the last integral in \eqref{lastint}, we compute the coefficient $I_8$ of $e^{8\pi i t}$ in the Fourier expansion of $\Re\big{(} \zeta_1^2(t) +\zeta_2^2(t)\big{)}$:
$$ I_8= \gamma ^2+\overline{2(\alpha,\beta)\cdot w^-_5 +2(\gamma,0)\cdot w^-_6 +2 w^0\cdot
w^-_4 +2w_1^- \cdot w^-_3 +w^-_2\cdot w^-_2}.$$
We want to find a constant $c$ such that
\begin{equation}\label{c3}
\pi( |\alpha|^2 +|\beta|^2 +2|\gamma|^2)- \frac c2 ( |\alpha|^2 +|\beta|^2 +|\gamma|^2) -\| w^-\|_E -\frac c2 \|w^0 +w^-\|^2_{L^2}- \frac c4 |I_8| < 0 .
\end{equation}
Applying the Cauchy-Schwarz inequality, we get
\begin{equation*}
\begin{split}
|I_8|\,\ge &|\gamma|^2 -a_1|w^-_5|^2-\frac 1{a_1} |(\alpha,\beta)|^2 -a_2|w^-_6|^2 -\frac 1{a_2} |\gamma|^2\\&-|w^0|^2-|w^-_4|^2-|w^-_1|^2-|w^-_3|^2-|w^-_2|^2 .
\end{split}
\end{equation*}
Choosing $a_1$ such that $\frac c4 a_1 =\frac c2 +5\pi$ and $a_2$ such that $ \frac c4 a_2= \frac c2 +6\pi$, equation \eqref{c3} becomes
\begin{equation}\label{c3-1}
 (|\alpha|^2+ |\beta|^2)\bigg{(}\pi-\frac c2 - \frac{(\frac c4)^2}{\frac c2 +5\pi}\bigg{)} +|\gamma|^2 \bigg{(}2\pi-\frac c2 -\frac c4 +\frac{(\frac c4)^2}{\frac c2 +6\pi }\bigg{)}+ \dots,
\end{equation}
where dots indicate other negative terms not involving $\alpha,\beta$ or $\gamma$.
The expression in \eqref{c3-1} is negative if
\begin{equation}\label{sist}
\begin{cases}
\pi-\frac c2 - \frac{(\frac c4)^2}{\frac c2 +5\pi} <0 \\
2\pi-\frac c2 -\frac c4 +\frac{(\frac c4)^2}{\frac c2 +6\pi }<0.
\end{cases}
\end{equation}
Solving the system \eqref{sist}, we obtain
$$ c>4\pi \frac{\sqrt{109}-7}{5} .$$
Considering the approximating domains $\mathcal{D}_n$ and reasoning as in the last part of the proof of Thorem \ref{c2bis}, we conclude that 
\begin{equation}\label{d1}
c_3(D^2\times D^2)<4\pi \frac{\sqrt{109}-7}{5}.
\end{equation} Moreover, Proposition \ref{c32} implies
\begin{equation}\label{d2}
c_3(D^2\times D^2)>c_3(\mathbb{B}^2)=2\pi.
\end{equation} Looking at the spectrum $\Sigma(D^2\times D^2)$ computed in Proposition \ref{s1}, we see that \eqref{d1} and \eqref{d2} imply $c_3(D^2\times D^2)=8.$ 
\end{proof}

Note that in the proof of Theorem \ref{cap3} we have used the fact that the capacity $c_3$ of the real bidisc is strictly greater than the corresponding capacity for the unit ball $\mathbb{B}^2$. This inequality is proved below in Proposition \ref{c32} for every capacity $c_k$. The proof relies on the following lemma. 

\begin{lemma}\label{c31} Let $D_1\subset D_2$ be two convex smooth subdomains of $\C^n$ such that for some $k$ we have $c_k(D_1)=c_k(D_2)=:c$. Assume also that $c$ is isolated in $\Sigma(D_1)$ and $\Sigma(D_2)$. Then there exists a closed characteristic $\gamma\subset \partial D_1\cap\partial D_2$ such that $\mathcal{A}(\gamma)=c$.
\end{lemma}
\begin{proof} Let $r_1$ and $r_2$ be the gauge functions of $D_1$ and $D_2$ respectively. Let $f$ be an increasing positive function such that
\begin{equation}
f(s)=\begin{cases}
0 \,\,\,\,\,\,\,\text{ if }s\le 1 \\
Cs \,\,\,\text{ if $s$ is large}
\end{cases}
\end{equation}
and $f'(s_0)=c$ only for one $s_0 >1$. We will specify later how $s_0$ and $C$ are chosen. Let $H^f_i:=f\circ r_i$ for $i=1,2$. Note that 
\begin{equation}\label{equa}
c_k(D_1)=\inf_{H_1}\{ c^k_{H_1} \}=\inf_{H_2}\{ c^k_{H_2} \} =c_k(D_2),
\end{equation} where each infimum is taken over all possible admissible Hamiltonians $H_i$ for $D_i$. 
Equation \eqref{equa} together with the fact that $c$ is isolated in the spectra $\Sigma(D_1)$ and $\Sigma(D_2)$ implies that, choosing a function $f$ with $C$ sufficiently large, we obtain
$$ c^k_{H^f_1} =c^k_{H^f_2}.$$
Let $\mathcal{A}_{H^f_i}$ be the corresponding Hamiltonian actions. Since $D_1\subset D_2$, then $r_1\ge r_2$ and $\mathcal{A}_{H^f_1} \le\mathcal{A}_{H^f_2}$.
Let $$W_{\epsilon}=\big{\{} \zeta\in E\, |\, \mathcal{A}_{H^f_2}(\zeta)\le c^k_{H^f_2}+\epsilon \big{\}}.$$ $W_{\epsilon}$ is a closed equivariant set of pseudoindex at least $k$.
We have
$$ c^k_{H^f_1}\le \sup_{\zeta\in W_\epsilon}\{ \mathcal{A}_{H^f_1}(\zeta)\}\le   
\sup_{\zeta\in W_\epsilon}\{ \mathcal{A}_{H^f_2}(\zeta)\}=c^k_{H^f_2}+\epsilon .$$
We now claim that there exist two sequences $\epsilon_n \to 0$ and $\zeta_n\in W_{\epsilon_n}$ such that 
$\nabla_E\mathcal{A}_{H^f_1}(\zeta_n)\to 0$ and $\mathcal{A}_{H^f_1}(\zeta_n)\to c^k_{H^f_1}$.  Indeed, assume that this is not true. Then there exist $\epsilon_0 >0$ and $\delta_0>0$ such that for $\epsilon\le\epsilon_0$ and for all $\zeta\in W_\epsilon$ such that $|\mathcal{A}_{H^f_1}(\zeta)-c^k_{H^f_1}| \le \varepsilon_0$ we have $\| \nabla\mathcal{A}_{H^f_1}(\zeta)\|_E >\delta_0$. Folllowing \cite[Lemma 1]{EH90} and \cite[Proposition 2]{EH89}, 
we denote by $\Phi_t$ be the flow at the time $t$ in $E$ of the vector field $-\nabla\mathcal{A}_{H^f_1}$ with a suitable cut-off. Choosing $\epsilon$ small enough we have, for $t>\frac\epsilon {\delta_0^2}$, that 
$$ \sup_{\zeta \in \Phi_t (W_\epsilon)} \mathcal{A}_{H^f_1}(\zeta) <c^k_{H^f_1},$$
which gives a contradiction.

We have thus proved that there exists a subsequence of $\zeta_n$ converging in $E$ to $\zeta_0$, which is a critical point for $\mathcal{A}_{H^f_1}$ and such that $\mathcal{A}_{H^f_1}(\zeta_0)=c^k_{H^f_1}$. 
If the image of $\zeta_0$ is not in $\partial D_1\cap \partial D_2$, then $r_1(\zeta_0(t))\ge r_2(\zeta_0(t))$ and the inequality is strict in some open interval.  
This implies that $\mathcal{A}_{H^f_2}(\zeta_0)>c^k_{H^f_2}$ and in turn $\mathcal{A}_{H^f_2}(\zeta_n)>c^k_{H^f_2}+\epsilon$ for some $\epsilon>0$ not depending on $n$. 
But this is in contradiction with the definition of $W_{\varepsilon_n}$.
\end{proof}
\bp \label{c32} For the real bidisc $D^2\times D^2$ and the unit ball $\mathbb{B}^2$ in $\C^2$, we have $c_k(D^2\times D^2)>c_k(\mathbb{B}^2)$ for every positive integer $k$. \ep
\begin{proof} Since the unit ball $\mathbb{B}^2$ centered at $0$ is contained in $D^2\times D^2$ we have, for any $k$, that $c_k(D^2\times D^2)\ge c_k(\mathbb{B}^2)$. Choose a smooth domain $W$ containing $\mathbb{B}^2$ and contained in $D^2\times D^2$ with a 
discrete action spectrum and having the same intersection at the bounday with $D^2\times D^2$ and $\mathbb{B}^2$. If $c_k(D^2\times D^2)=c_k(\mathbb{B}^2)$ for some $k$, then Lemma \ref{c31} implies that there exists a characteristic contained in 
\begin{equation*}
\begin{split}
\partial W\cap \mathbb{B}^2&=\partial (D^2\times D^2)\cap \mathbb{B}^2\\&=\{(x_1,x_2)\in\C^2 \,\vert\, x_1^2+x^2_2=1\}\cup\{(iy_1,iy_2)\in\C^2\,\vert\, y^2_1+y^2_2=1 \}.
\end{split}
\end{equation*} Since there are no characteristics in the intersection of the boundaries, 
then we must have $c_k(D^2\times D^2)>c_k(\mathbb{B}^2)$. 
\end{proof}

\begin{remark}\label{R2}
Let $B(a)=E(a,a)$, that is, $B(a)$ is the Euclidean ball of radius $\sqrt{a/\pi}$. It follows from \cite{R17} that the real bidisc $D^2\times D^2$ can be symplectically embedded into the ellipsoid $E(4,3\sqrt{3})$ and that the ball $B(4)$ can be symplectically embedded into $D^2\times D^2$. This implies that $c_3(D^2\times D^2)=8$. Proposition \ref{c32} can also be obtained from \cite{R17}. Indeed, $\mathbb{B}^2=B(\pi)$ has strictly smaller Ekeland-Hofer capacities than $B(4)$, which embeds into $D^2\times D^2$.
\end{remark}

\section{Applications}\label{SectionMain}

We now exploit our computations of the capacities of $D^2\times D^2$ to prove some results of symplectic rigidity.
First recall that the complex bidisc $\Delta^2\subset\C^2$ is defined by 
$$\Delta^2:=\big{\{}(z_1,z_2)\in \C^2 | \  x_1^2+y^2_1 <1, x^2_2+y^2_2<1 \big{\}}.$$ 
 In \cite{ST12} 
Shukov and Tumanov applied techniques from classical complex analysis to prove that there exists no symplectic embedding of $D^2\times D^2$ into $\Delta^2$. Using symplectic capacities we can easily show that no symplectic embedding is possible in the other direction.
\begin{corollary}\label{mainn1} There is no symplectic embedding of $\Delta^2$ into $D^2\times D^2$.
\end{corollary}
\begin{proof}
Assume by contradiction that there is such an embedding $\psi :\Delta^2\rightarrow D^2\times D^2$. By the extension after restriction principle, for any $\epsilon>0$ there exists a symplectic map with compact support $\psi_\epsilon :\C^2 \rightarrow \C^2$ such that $\psi_\epsilon |_{\Delta^2_{1-\epsilon}} =\psi |_{\Delta^2_{1-\epsilon}}$. Therefore
$$3\pi (1-\epsilon)^2 =c_3(\Delta^2_{1-\epsilon})=c_3(\psi_{\epsilon}(\Delta^2_{1-\epsilon}))\le c_3 (D^2\times D^2)=8,$$
which gives a contradiction.
\end{proof}
Note that the proof of Corollary \ref{mainn1} implies that the complex bidisc cannot be embedded even in a slightly larger real bidisc.

\begin{remark} Corollary \ref{mainn1} can also be obtained without using the Ekeland-Hofer capacities. By \cite{R17}, the bidisc $D^2\times D^2$ is a concave toric domain. One can then apply \cite[Theorem 1.18]{GH18} to conclude that the cube capacity $c_{\Box}$ of $D^2\times D^2$ is equal to $2$. By the very definition of the cube capacity \cite[Definition 1.17]{GH18}, Corollary \ref{mainn1} follows. 
\end{remark}

To prove the next rigidity result we need to recall a product property of the Ekeland-Hofer capacities: if $A\subset \C^n $ and $B\subset \C^m$, then
$$ c_k(A\times B)=\min_{i+j=k}\{c_i(A)+c_j(B)\}.$$
Here we use the convention that the zero-th capacity is equal to $0$, that is, $c_0(A)=c_0(B)=0$.
We denote by $\Delta_R$ the standard complex disc of radius $R$.
\begin{corollary}\label{3.1}
The product $D^2\times D^2\times \Delta_R$ is not symplectomorphic to $\Delta^2\times \Delta_R$ for $R>\sqrt{\frac{3\sqrt{3}}{2\pi}}$.
\end{corollary}
\begin{proof}
The case $R\ge 1$ is known \cite[Theorem 4.1]{W18}, hence let $R<1$. We have
$c_2(\Delta^2\times \Delta_R)=2\pi R^2$. On the other hand, $c_2(D^2\times D^2\times \Delta_R)=\min\{ 3\sqrt{3},4+\pi R^2, 2\pi R^2\}$.
The two capacities are different if $R>\sqrt{\frac{3\sqrt{3}}{2\pi}}$.
\end{proof}
\begin{remark}
The bound in Corollary \ref{3.1} can be improved to $\sqrt{2/\pi}$ using Ekeland-Hofer symplectic capacities of higher order. More precisely, one has to show, for each positive integer $n$, that $c_{2n-1}(D^2\times D^2)=4n$. This can be achieved by arguing in a similar way as in Theorem \ref{cap3}, where we computed the value of $c_3(D^2\times D^2)$. Vinicius Gripp Barros Ramos has informed the authors that the bound $\sqrt{2/\pi}$ can also be obtained using the results in \cite{GH18}.
\end{remark}

\section{Acknowledgements}
The authors would like to warmly thank Vinicius Gripp Barros Ramos for insightful comments on the first version of this paper. We also acknowledge helpful suggestions and remarks from the anonymous referee. 

 							%	\bibliography{mybibTVK}

 							\end{document}